\documentclass[letterpaper,10pt,conference]{ieeeconf} 
\IEEEoverridecommandlockouts  
\usepackage{amsmath,amssymb,amsthm,epsfig,dsfont,color,subfigure,empheq,graphicx}
\usepackage{enumerate,url,wasysym,epstopdf,balance}
\usepackage[table]{xcolor}

\DeclareMathOperator{\nullspace}{null}

\DeclareMathOperator{\sign}{sgn}

\newtheorem{lemma}{Lemma}
\newtheorem{corollary}{Corollary}
\newtheorem{theorem}{Theorem}
\newtheorem{definition}{Definition}
\newtheorem{assumption}{Assumption}

\theoremstyle{remark}\newtheorem{remark}{Remark}
\allowdisplaybreaks

\newcommand \bn{\mathbf{n}}

\newcommand \bx{\mathbf{x}}

\newcommand \bq{\mathbf{q}}
\newcommand \bA{\mathbf{A}}

\newcommand \bzero{\boldsymbol{0}}

\newcommand \tbphi{\tilde{\boldsymbol{\phi}}}
\newcommand \tbpsi{\tilde{\boldsymbol{\psi}}}

\newcommand \hbphi{\hat{\boldsymbol{\phi}}}
\newcommand \hbpsi{\hat{\boldsymbol{\psi}}}

\newcommand \mcC{\mathcal{C}}

\newcommand \mcG{\mathcal{G}}

\newcommand \mcM{\mathcal{M}}
\newcommand \mcN{\mathcal{N}}
\newcommand \mcP{\mathcal{P}}

\newcommand \mcS{\mathcal{S}}
\newcommand \mcT{\mathcal{T}}

\newcommand \bmcP{\bar{\mathcal{P}}}

\newcommand \bphi{\boldsymbol{\phi}}
\newcommand \bpsi{\boldsymbol{\psi}}

\title{\LARGE \bf Natural Gas Flow Equations: Uniqueness and an MI-SOCP Solver}

\author{Manish K. Singh and Vassilis Kekatos
}

\begin{document}

\maketitle
\thispagestyle{empty}
\pagestyle{empty}

\begin{abstract}
The critical role of gas fired-plants to compensate renewable generation has increased the operational variability in natural gas networks (GN). Towards developing more reliable and efficient computational tools for GN monitoring, control, and planning, this work considers the task of solving the nonlinear equations governing steady-state flows and pressures in GNs. It is first shown that if the gas flow equations are feasible, they enjoy a unique solution. To the best of our knowledge, this is the first result proving uniqueness of the steady-state gas flow solution over the entire feasible domain of gas injections. To find this solution, we put forth a mixed-integer second-order cone program (MI-SOCP)-based solver relying on a relaxation of the gas flow equations. This relaxation is provably exact under specific network topologies. Unlike existing alternatives, the devised solver does not need proper initialization or knowing the gas flow directions beforehand, and can handle gas networks with compressors. Numerical tests on tree and meshed networks with random gas injections indicate that the relaxation is exact even when the derived conditions are not met.
\end{abstract}


\section{Introduction}
Natural gas has been a critical energy source for decades with uses across the residential, commercial, industrial, and electric generation sectors~\cite{mercado2015review}. The importance of natural gas in the energy sector has further increased, and the same is anticipated in future. The main reasons for increasing emphasis on natural gas include the discovery of substantial new supplies of natural gas in the U.S., and its recognition as a clean low-carbon solution to meet rising energy demands~\cite{MIT}. The thrust for renewable energy in power sector demands for technical solutions to handle the high variability, intermittency, and uncertainty involved with wind and solar generations. Natural gas has come up as an economically viable and low-carbon solution to the said problem because of high ramping abilities of natural gas-fired generators~\cite{MIT}.

The primary transportation mode for natural gas is through a large continent-wide network of pipelines~\cite{mercado2015review}. Along a pipeline, pressure drops in the direction of flow due to friction. Gas contracts necessitate the operators to maintain a minimum pressure at consumer nodes. Therefore, compressors are placed on some pipelines to increase the pressure at the output. Given gas injection/withdraws, operators need to solve the \emph{gas flow} (GF) equations, a set of nonlinear equations governing the distribution of gas flows and nodal pressures~\cite{abhi2017GM}. The increasing variability in gas withdrawals by gas-fired generators, the complex interdependence of gas-electric infrastructure, and increased focus on reliability, they all motivate well efficient GF solvers~\cite{esquivel2012unified}. Solving the GF problem is hard for non-tree networks even under-steady state and balanced conditions \cite{Dvijotham15}.

The GF problem is typically solved using the Newton-Raphson (NR) scheme, though its convergence is conditioned on proper initialization~\cite{esquivel2012unified}. A semidefinite program (SDP)-based GF solver, attaining a higher success probability than the NR scheme, is developed in \cite{abhi2017GM}. Nevertheless, the SDP-based solver fails to solve the GF problem if the network state is far from the states considered in designing the solver. The necessity of good initialization for the NR scheme and suitable design points for SDP-based solver limit their suitability for reliability studies. For simpler networks without compressors, the flows and pressures are the optimal primal-dual solutions of a convex minimization~\cite{wolf2000energy}. Broadening the scope to tree networks, a GF solver based on the monotone operator theory has been developed in~\cite{Dvijotham15}. The latter applies to meshed GNs presuming that the directions of gas flows are given. Reference~\cite{Dvijotham15} establishes also the uniqueness of a GF solution but only within a monotonicity domain; still this domain is hard to characterize for non-tree networks or networks with compressors.

The contribution of this work is on three fronts: First, Section~\ref{sec:unique} proves that the nonlinear steady-state GF equations enjoy a unique solution, if a solution exists. To the best of our knowledge, this is the first claim corroborating the numerical observations of~\cite{Dvijotham15}. Second, Section~\ref{subsec:cr:problem} puts forth an MI-SOCP-based GF solver. The GF task is posed as a minimization problem where flow directions are captured by binary variables; the nonlinear GF equations are relaxed to second-order cone constraints; and a judiciously designed function is appended in the objective. Thanks to the latter, the third contribution is to show that the relaxation is exact if there are no compressors in cycles and every pipe belongs to at most one cycle. Numerical tests on tree and meshed networks demonstrate that the devised solver finds the unique GF solution even when the assumed conditions are not met. 

\emph{Notation}: lower- (upper-) case boldface letters denote column vectors (matrices). Calligraphic symbols are reserved for sets. Symbol $^{\top}$ stands for transposition, and $\odot$ denotes entry-wise multiplication between vectors. Vectors $\mathbf{0}$ and $\mathbf{1}$ are the all-zero and all-one vectors. The sign function $\sign(x)$ returns $+1$ if $x>0$; $-1$ if $x<0$; and $0$ otherwise. 

\section{Gas Flow Problem}\label{sec:model}
Consider a natural gas network (GN) modeled by a directed graph $\mathcal{G}=(\mcN,\mcP)$. The graph vertices $\mcN=\{1,\cdots,N\}$ model nodes where gas is injected or withdrawn from the network, or simple junctions. The graph edges $\mcP=\{1,\ldots,P\}$ correspond to gas pipelines connecting two nodes. Let $p_n>0$ be the gas pressure at node $n$ for all $n\in\mcN$. One of the nodes (conventionally one hosting a large gas producer) is selected as the reference node. The reference node is indexed by $r$, and its pressure is fixed to a known value $p_r$. The gas injection $q_n$ at node $n\in\mcN$ is positive for an injection node; negative for a withdrawal node; and zero for junction nodes. 

Without loss of generality, edges are assigned an arbitrary direction denoted by $\ell=(m,n)\in\mcP$ with $m$ and $n\in\mcN$. The gas flow $\phi_\ell$ on pipeline $\ell=(m,n)\in\mcP$ is positive when gas flows from node $m$ to node $n$, and negative, otherwise. Conservation of mass implies that for all $n\in\mcN$
	\begin{equation}\label{eq:mc}
	q_n = \sum_{\ell:(n,k) \in \mcP}\phi_\ell - \sum_{\ell:(k,n) \in \mcP}\phi_\ell.
	\end{equation}
To express \eqref{eq:mc} in matrix-vector form, collect all gas injections in $\bq:=[q_1\dots q_N]^\top$ and edge flows in $\bphi:=[\phi_1\dots\phi_P]^\top$.The connectivity of the GN graph is captured by the $P\times N$ edge-node incidence matrix $\mathbf{A}$ with entries
	\begin{equation}\label{eq:A}
	A_{\ell,k}:=
	\begin{cases}
	+1&,~k=m\\
	-1&,~k=n\\
	0&,~\text{otherwise}
	\end{cases}~\forall~\ell=(m,n)\in\mcP.
	\end{equation}
The mass conservation in \eqref{eq:mc} may thus be expressed as
\begin{equation}\label{eq:mc2}
\bA^\top\bphi=\bq.
\end{equation}
Observe that $\bA\mathbf{1}=\mathbf{0}$ by definition, and so $\mathbf{1}^\top\bq=0$. The latter is intuitive since the network should be balanced under steady-state conditions. This also implies that \eqref{eq:mc} provides $N-1$ rather than $N$ independent linear equations on $\bphi$. 

For high- and medium-pressure networks, the pressure drop and energy loss across a pipeline are captured by a set of partial differential equations evolving across time and spatially along the pipeline length~\cite{Osiadacz87}, \cite{ThorleyTiley87}. Ignoring friction, pipeline tilt, and assuming time-invariant gas injections, this set of partial differential equations simplifies to the so termed Weymouth equation~\cite{Wu99}
\begin{equation}\label{eq:weymouth}
	p_m^{2} - p_n^{2} = a_\ell\phi_\ell |\phi_\ell|
\end{equation}
describing the pressure difference across the endpoints of pipeline $\ell=(m,n)\in\mcP$. The parameter $a_\ell>0$ depends on the physical properties of the pipeline~\cite{Osiadacz87}. The Weymouth equation asserts that pressure drops within a pipeline in the direction of gas flow. To be precise, the difference of squared pressures is proportional to the squared gas flow. To simplify notation, define the squared pressure at node ${n\in\mcN}$ as ${\psi_n:=p_n^2}$. Then, equation \eqref{eq:weymouth} can be written as
\begin{subequations}\label{eq:wey2}
	\begin{align}
	\psi_m-\psi_n&=a_\ell\sign(\phi_\ell)\phi_\ell^2\quad&&,~\forall \ell=(m,n)\in\mcP\label{seq:wey2a}\\
	\psi_n&\geq0&&,~\forall n\in\mcN.\label{seq:wey2b}
	\end{align}
\end{subequations}
By slightly abusing the terminology, we will oftentimes refer to $\psi_m$ as pressure rather than squared pressure, but the distinction will be clear from the context. 

To avoid unacceptably low or high pressures, GN operators install compressors at selected pipelines comprising the set $\mcP_a\subseteq \mcP$ with $P_a=|\mcP_a|$. Its complement set $\bmcP_a:=\mcP\setminus\mcP_a$ includes the remaining lossy pipelines satisfying \eqref{eq:weymouth}--\eqref{eq:wey2}. A compressor amplifies the squared pressure between its input and output by a ratio $\alpha_\ell$. Moreover, a compressor allows a unidirectional flow in the direction of compression. Therefore, compressor $\ell=(m,n)\in\mcP_a$ can be modeled as
\begin{subequations}\label{eq:comp}
	\begin{align}
	\psi_n&=\alpha_\ell\psi_m\label{seq:compa}\\
	\phi_\ell&\geq0.\label{seq:compb}
	\end{align}
\end{subequations}
Equation \eqref{eq:comp} assumes an ideal (lossless) compressor, that is $a_\ell=0$. This is wlog since an actual non-ideal compressor on pipe $(m,n)$ can be modeled by inserting an additional node $n'$ between nodes $m$ and $n$: Then, the lossless pipe $(m,n')$ hosts an ideal compressor and pipe $(n',n)$ is lossy. Both pipes have identical flows $\phi_{mn'}=\phi_{n'n}$.

Based on \eqref{eq:wey2}--\eqref{eq:comp}, it is not hard to verify that an NG network can be uniquely described by either the vector of nodal pressures $\bpsi$, or the vector of edge flows $\bphi$. 

\begin{lemma}\label{le:fp}
Given a reference squared pressure $\psi_r$ for some $r\in\mcN$, a pair $(\bphi,\bpsi)$ satisfying \eqref{eq:mc2}, \eqref{eq:wey2}, and \eqref{eq:comp}, is uniquely described by either $\bphi$ or $\bpsi$. 
\end{lemma}

In fact, finding one of these two vectors constitutes the \emph{gas flow} (GF) problem formally described as follows. 

\begin{definition}
Given the pressure $\psi_r=p_r^2$ at the reference node; balanced nodal injections $\bq$; the compression ratios $\alpha_\ell$ for all compressors $\ell\in\mcP_a$; and the friction parameters $a_\ell$ for all lossy pipes $\ell\in\bmcP_a$, the GF problem aims at finding the nodal pressures $\bpsi$ and pipe flows $\bphi$ satisfying the GF equations \eqref{eq:mc2}, \eqref{eq:wey2}, and \eqref{eq:comp}. 
\end{definition}

The task involves $N-1+P$ equations over $N-1+P$ unknowns. The GF task can be posed as the feasibility task
\begin{align*}\label{eq:G1}
\mathrm{find}~&~\{\bphi,\boldsymbol{\psi}\}\tag{G1}\\
\mathrm{s.to}~&~ \eqref{eq:mc2},\eqref{eq:wey2},\eqref{eq:comp}.
\end{align*}
Albeit \eqref{eq:mc2} and \eqref{eq:comp} are linear, the Weymouth equation in \eqref{eq:wey2} is piecewise quadratic and non-convex. In addition, the requirement $\{\phi_\ell \geq 0\}_{\ell\in\mcP_a}$ could further complicate solving the GF equations. The GF task is typically solved using the Newton-Raphson's scheme, which converges only if initialized sufficiently close to a solution~\cite{Martinez11}. 
	
If the GN is a tree $(P=N-1)$, then $\mathbf{A}$ is full row-rank and \eqref{eq:mc2} is invertible. Once the pipe flows $\bphi$ are known, pressures $\boldsymbol{\psi}$ can be found through \eqref{eq:wey2}--\eqref{eq:comp}. In practice though, gas networks can exhibit non-radial structure~\cite{mercado2015review},~\cite{Wu99}. Therefore, solving the GF problem remains non-trivial. Before devising a new GF solver in Section~\ref{sec:cr}, the next section establishes that the GF equations enjoy a unique solution.

\section{Uniqueness of the GF Solution}\label{sec:unique}
The analysis requires some concepts from graph theory, which are briefly reviewed next. A directed graph $\mcG=(\mcM, \mcP)$ is connected if there exists a sequence of adjacent edges between any two of its nodes. All graphs in this work are assumed to be connected. A minimal set of edges $\mcP_\mcT$ preserving the connectivity of a graph constitutes a \emph{spanning tree} of $\mcG$; is denoted by $\mcT:=(\mcM,\mcP_\mcT)$; and $|\mcP_\mcT|=|\mcM|-1$. The edges not belonging to a spanning tree $\mcT$ are referred to as \emph{links} with respect to $\mcT$. 

\begin{figure}[t]
	\centering
	\includegraphics[scale=0.65]{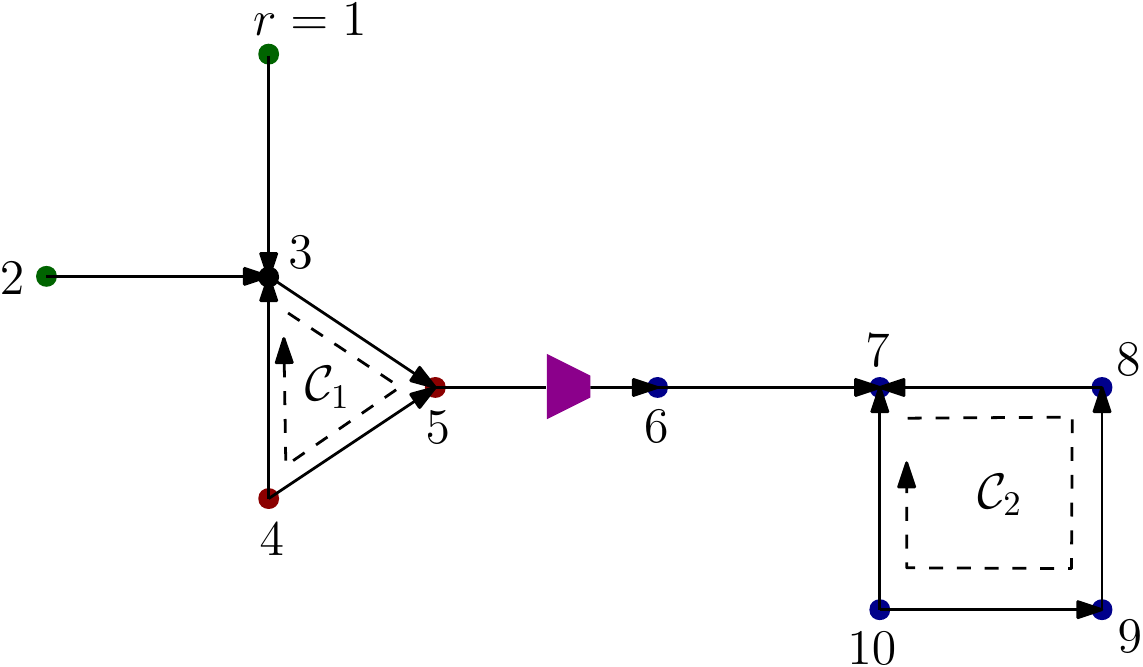}
	\caption{Example gas network with $11$ edges, $2$ cycles, and $1$ compressor.}
	\label{fig:example}
\end{figure}

A \emph{path} between nodes $m$ and $n$ is defined as a sequence of adjacent edges between the two nodes. If a path starts and ends at the same node (without edge or node repetition), it constitutes a \emph{cycle}. A \emph{tree} is a connected graph with no cycles. The statements `cycle $\mcC$ contains node $i$' or `node $i$ belongs to cycle $\mcC$' mean that there exists an edge in $\mcC$ that is incident to node $i$. For any cycle $\mcC$, we can select an arbitrary direction and define its indicator vector $\bn^\mcC$ whose $\ell$-th entry is
\begin{equation}\label{eq:nc}
n_{\ell}^\mcC:=
\begin{cases}
0&,~\text{if edge } \ell\notin \mcC\\
+1&,~\text{if direction of }\ell\text{ agrees with cycle direction}\\
-1&,~\text{otherwise.}\\
\end{cases}
\end{equation}
For example, the network of Fig.~\ref{fig:example} has $11$ edges and $2$ cycles. Based on the assigned direction for $\mcC_2$, the entries of $\bn^{\mcC_2}$ corresponding to edges $\left\{(7,8),(8,9),(9,10), (10,11)\right\}$ are $\{-1,-1,-1,+1\}$; the remaining entries are zero. Given a spanning tree, a \emph{fundamental cycle} is a cycle formed by adding a link to the spanning tree. 

After the graph preliminaries, we proceed with the uniqueness of the GF solution. The proof relies on the next result for a single-cycle graph, which is proved in the appendix.
 
\begin{lemma}\label{le:1cycle}
Consider a graph comprising a single cycle $\mcC$ over $k+1$ nodes indexed by $i=0,\dots, k$, and define $\bn^\mcC\in\mathbb{R}^k$ as the indicator vector for this cycle. For fixed squared pressure $\psi_0$, if two flow vectors $\bphi$ and $\tilde{\bphi}$ with $\tilde{\bphi}\neq \bphi$ satisfy \eqref{eq:wey2}--\eqref{eq:comp}, they cannot satisfy 
	\begin{align*}
	\sign(\tilde{\bphi}-\bphi)\odot\bn^\mcC&<\mathbf{0};~\textrm{or}~\\
	\sign(\tilde{\bphi}-\bphi)\odot\bn^\mcC&>\mathbf{0}
	\end{align*}
where the inequalities are understood entrywise. 	
\end{lemma} 

The proof of Lemma~\ref{le:1cycle} exploits the fact that the pressure differences across a loop should sum up to zero no matter what the flows $\bphi$ or $\tilde{\bphi}$ are. The usefulness of this lemma is to eliminate scenarios of flow changes. Take for example the case where $\bn^\mcC=\mathbf{1}$, and assume two flow vectors $\bphi$ and $\tilde{\bphi}$ satisfying \eqref{eq:wey2}--\eqref{eq:comp}. We hypothesize the flows in $\tilde{\bphi}$ are larger than the flows in $\bphi$, that is $\tilde{\bphi}>\bphi$. Lemma~\ref{le:1cycle} ensures that this hypothesis is not valid since $\sign(\tilde{\bphi}-\bphi)\odot\bn^\mcC=\mathbf{1}\odot\mathbf{1}=\mathbf{1}>\mathbf{0}$. The hypothesis $\tilde{\bphi}<\bphi$ can be crossed out likewise. 

The proof for the uniqueness of the GF solution builds also on the ensuing result known from network flows. Its proof follows as a special case of \cite[Th.~8.8]{korte2012combinatorial} by setting the source-destination flow to zero. 

\begin{lemma}\label{le:MIT}
Suppose $\bA$ is the edge-node incidence matrix of a directed graph. For any vector $\bn\in\nullspace(\bA^\top)$, there exists a set of cycles $\mcS_\mcC$ such that
\begin{equation}\label{eq:MIT}
\bn=\sum_{\mcC\in\mcS_\mcC}\lambda_\mcC\bn^\mcC
\end{equation}
where $\lambda_\mcC\geq0$ for all $\mcC\in\mcS_\mcC$, and $\bn^\mcC\odot\bn^{\mcC'}\geq \mathbf{0}$ for all $\mcC,\mcC'\in\mcS_\mcC$.
\end{lemma}

Lemma~\ref{le:MIT} asserts that any vector in $\nullspace(\bA^\top)$ can be expressed as a conic combination of cycle indicator vectors. Moreover, if any pair of these cycles shares an edge, this edge participates to both cycles in the same direction. The following claim follows easily from \eqref{eq:MIT}.

\begin{corollary}\label{co:MIT}
If $\lambda_\mcC>0$ in the representation of \eqref{eq:MIT}, then $n_\ell \cdot n_\ell^\mcC>0$ for all $\ell\in \mcC$.
\end{corollary}

Using Lemmas~\ref{le:1cycle}--\ref{le:MIT} and Corollary~\ref{co:MIT}, we next prove the uniqueness of the solution to the steady-state GF problem.

\begin{theorem}\label{th:GFunique}  
The gas flow problem \eqref{eq:G1} has a unique solution, if feasible.
\end{theorem}

\begin{proof}
Proving by contradiction, assume $\bphi$ and $\tilde{\bphi}$ are two distinct flow vectors solving \eqref{eq:G1}. Since both vectors satisfy \eqref{eq:mc2}, their difference $\bn:=\bphi-\tilde{\bphi}$ must lie in $\nullspace(\bA^\top)$. Then, from Lemma~\ref{le:MIT}, vector $\bn$ can be decomposed as in \eqref{eq:MIT}. 

Select any edge $\ell$ corresponding to a non-zero entry of $\bn$. Since $n_\ell\neq 0$, there exists a cycle $\mcC$ for which $\lambda_\mcC \cdot n_\ell^\mcC\neq 0$. Two cases can be identified:

\emph{C1) Cycle $\mcC$ contains the reference node $r$.} From Corollary~\ref{co:MIT}, it follows that $n_\ell \cdot n_\ell^\mcC=(\phi_\ell-\tilde{\phi}_\ell)\cdot n_\ell^\mcC>0$ for all edges $\ell\in \mcC$. The latter contradicts Lemma~\ref{le:1cycle} and proves the claim. 

\emph{C2) Cycle $\mcC$ does not contain the reference node $r$.} Calculate the \emph{distance} of cycle $\mcC$ from $r$ that we define as
\begin{equation}\label{eq:distance}
d(\mcC):=\min_{i \in \mcC} d_{i-r} 
\end{equation}
where $d_{i-r}$ counts the edges in the shortest path between nodes $i$ and $r$. The node in $\mcC$ attaining distance $d(\mcC)$ will be indexed by $i_\mcC$. If multiple nodes satisfy the latter property, pick one arbitrarily. Consider the shortest path $i_\mcC-r$ between nodes $i_\mcC$ and $r$. Two cases can be identified again.

\emph{C2a) The entries of $\bn$ associated with the edges in path $i_\mcC-r$ are all zero.} This implies that the related entries in $\bphi$ and $\tilde{\bphi}$ are equal. Therefore, the nodal pressures along the path $i_\mcC-r$ can be recursively computed using \eqref{eq:wey2} and \eqref{eq:comp} starting from $\psi_r$. Then, the nodal pressures along $i_\mcC-r$ agree between $\boldsymbol{\psi}$ and $\tilde{\boldsymbol{\psi}}$, and so $\psi_{i_\mcC}=\tilde{\psi}_{i_\mcC}$. Since the pressure at node $i_\mcC$ has been fixed, this node can serve as a reference node, the argument under \emph{C1)} applies, and proves the claim. 

\emph{C2b) There exists an edge $\ell'$ in $i_\mcC-r$ with a non-zero entry in $\bn$.} This edge must belong to a cycle $\mcC'$. Check whether cycle $\mcC'$ falls under \emph{C1)} or \emph{C2)}, and apply the previous reasoning recursively.

The previous process considers cycles $\mcC$ with progressively smaller distances $d(\mcC)$, and so case \emph{C1)} eventually occurs. The recursion terminates upon finding a cycle contradicting Lemma~\ref{le:1cycle} under case \emph{C1)}, and completes the proof.
\end{proof}

\begin{remark}\label{re:parallel}
Two compressors can be sometimes connected in parallel~\cite{wolf2000energy}. If they are assumed ideal, it is not possible to infer the gas flow on each compressor. Under the practical assumption that the two compressors are non-ideal, this identifiability concern is waived as validated by Theorem~\ref{th:GFunique}.
\end{remark}

\section{MI-SOCP Relaxation}\label{sec:cr}
Having established the uniqueness of the GF solution, this section develops an MI-SOCP relxation of \eqref{eq:G1}, and then its exactness under some network conditions.

\subsection{Problem Reformulation}\label{subsec:cr:problem}
The non-convexity in \eqref{eq:G1} is due to the Weymouth equation in \eqref{seq:wey2a}. Using the big-$M$ trick and upon introducing a binary variable $x_\ell$ for every lossy pipe $\ell=(m,n)\in\bmcP_a$, constraint \eqref{seq:wey2a} can be relaxed to convex inequalities as
\begin{subequations}\label{eq:weyr}
	\begin{align}
	-&M(1-x_{\ell})\leq \phi_{\ell}\leq Mx_{\ell}\label{seq:weyra}\\
	-&M(1-x_{\ell})\leq \psi_m-\psi_n-a_{\ell}\phi_{\ell}^2\label{seq:weyrb}\\
	&\psi_m-\psi_n + a_{\ell}\phi_{\ell}^2\leq Mx_{\ell}\label{seq:weyrc}\\
	&x_{\ell}\in\{0,1\}\label{seq:weyrd}
	\end{align}
\end{subequations}
for some large $M>0$. When $x_\ell=1$, constraint \eqref{seq:weyra} yields $\phi_\ell\geq 0$, constraint \eqref{seq:weyrb} reads $\psi_m-\psi_n-a_{\ell}\phi_{\ell}^2\geq 0$, and \eqref{seq:weyrc} is satisfied trivially. When $x_\ell=0$, constraint \eqref{seq:weyra} yields $\phi_\ell\leq 0$, constraint \eqref{seq:weyrb} is satisfied trivially, and \eqref{seq:weyrc} becomes $\psi_m-\psi_n+a_{\ell}\phi_{\ell}^2\leq 0$. We therefore get that the binary variable equals $x_\ell=\sign(\phi_{\ell})$. It hence determines the direction of flow $\phi_\ell$, and activates \eqref{seq:weyrb} or \eqref{seq:weyrc}. 

By replacing \eqref{eq:wey2} with \eqref{eq:weyr} in \eqref{eq:G1}, the GF problem can be posed as an MI-SOCP. However, the minimizer of \eqref{eq:G2} is useful only if it satisfies \eqref{seq:weyrb} or \eqref{seq:weyrc} (depending on the value of $x_\ell$) with equality for all lossy pipes. Then, the relaxation is termed exact. Otherwise, the minimizer of \eqref{eq:G2} is infeasible for \eqref{eq:G1}. 

To promote exact relaxations, we will substitute the GF problem \eqref{eq:G1} by the minimization
\begin{align*}\label{eq:G2}
\min~&~r(\boldsymbol{\psi})\tag{G2}\\
\mathrm{over}~& ~ \bphi,\boldsymbol{\psi},\bx\\
\mathrm{s.to}~&~ \eqref{eq:mc2},\eqref{eq:comp},\eqref{eq:weyr}
\end{align*}
where vector $\bx$ contains all $x_\ell$ with $\ell\in\bmcP_{a}$, and
\begin{equation*}
r(\boldsymbol{\psi}):=\sum_{(m,n)\in \bmcP_{a}}|\psi_m-\psi_n|.
\end{equation*}
The cost $r(\boldsymbol{\psi})$ sums up the absolute pressure differences across all lossy pipes. A similar MI-SOCP relaxation has been proposed for optimizing flow in water distribution networks in \cite{singh2018optimal}. Although problem \eqref{eq:G2} remains non-convex due to the binary variables $\bx$, with advancements in MI-SOCP solvers, this minimization can be handled for moderately sized networks~\cite{vishnu2010misocp}, as corroborated by our numerical tests. The next section provides well-defined network conditions under which the relaxation \eqref{eq:weyr} in \eqref{eq:G2} is exact. 

\subsection{Exactness of the Relaxation}\label{subsec:cr:exact}
The relaxation from \eqref{eq:wey2} to \eqref{eq:weyr} has been proposed earlier in \cite{backhaus2016convex}, \cite{zlotnik2017adaptive}, \cite{zang2018IET}; yet without optimality guarantees. For instance, reference \cite{backhaus2016convex} solves gas expansion planning through this relaxation. Upon fixing the binary variables to the values obtained via the relaxation, the continuous variables are then heuristically refined by a local search. The convex relaxation has been combined with a McCormick relaxation in \cite{zlotnik2017adaptive}. This section provides network conditions ensuring that the relaxation of \eqref{eq:G2} is provably exact.

\begin{assumption}\label{as:A1}
Graph $\mcG$ has no compressors in cycles.
\end{assumption}

\begin{assumption}\label{as:A2}
Every edge of $\mcG$ belongs to at most one cycle. 
\end{assumption}

\begin{theorem}\label{th:gfexact}
 Under Assumptions~\ref{as:A1} and \ref{as:A2}, every minimizer of \eqref{eq:G2} solves the GF problem \eqref{eq:G1}, if the latter is feasible.
\end{theorem}

Heed that Assumptions 1 and 2 may not be always satisfied in practical GNs; see e.g., Remark~\ref{re:parallel}. Albeit, the tests of Section~\ref{sec:tests} show that \eqref{eq:G2} solves \eqref{eq:G1} even when Assumption~\ref{th:gfexact} does not hold. Either way, Theorem~\ref{th:gfexact} guarantees that \eqref{eq:G2} can provably handle the GF task for a broader class of GNs than existing alternatives; recall \cite{wolf2000energy} cannot handle compressors, and \cite{Dvijotham15} presumes known flow directions. 

\begin{figure}[t]
	\centering
	\includegraphics[scale=0.34,angle=-90]{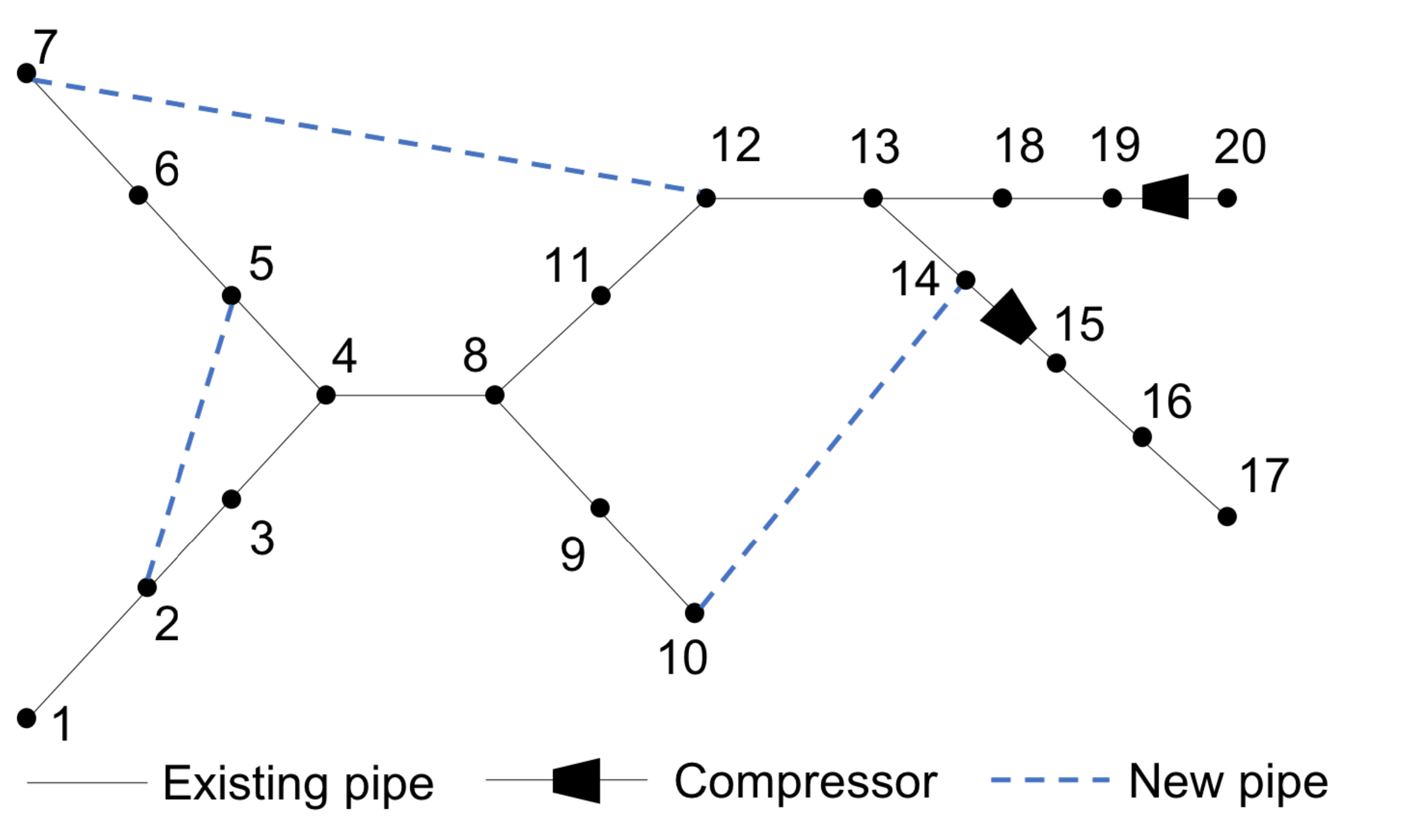}
	\caption{Modified Belgian natural gas network.}
	\label{fig:belgian}
\end{figure}

\section{Numerical Tests}\label{sec:tests} 
Our relaxed GF solver was tested on the Belgian benchmark network shown in Figure~\ref{fig:belgian}. Parallel pipes were replaced by their equivalents, and the compression ratios were determined based on the nodal pressures found in~\cite{c}. Problem \eqref{eq:G2} was solved using the MATLAB-based optimization toolbox YALMIP along with the mixed-integer solver CPLEX~\cite{yalmip},~\cite{cplex}, on a 2.7~GHz Intel Core i5 computer with 8~GB RAM. For the big--$M$ trick, we set $M=10^4$.

\begin{figure}[t]
	\centering
	\includegraphics[scale=0.2]{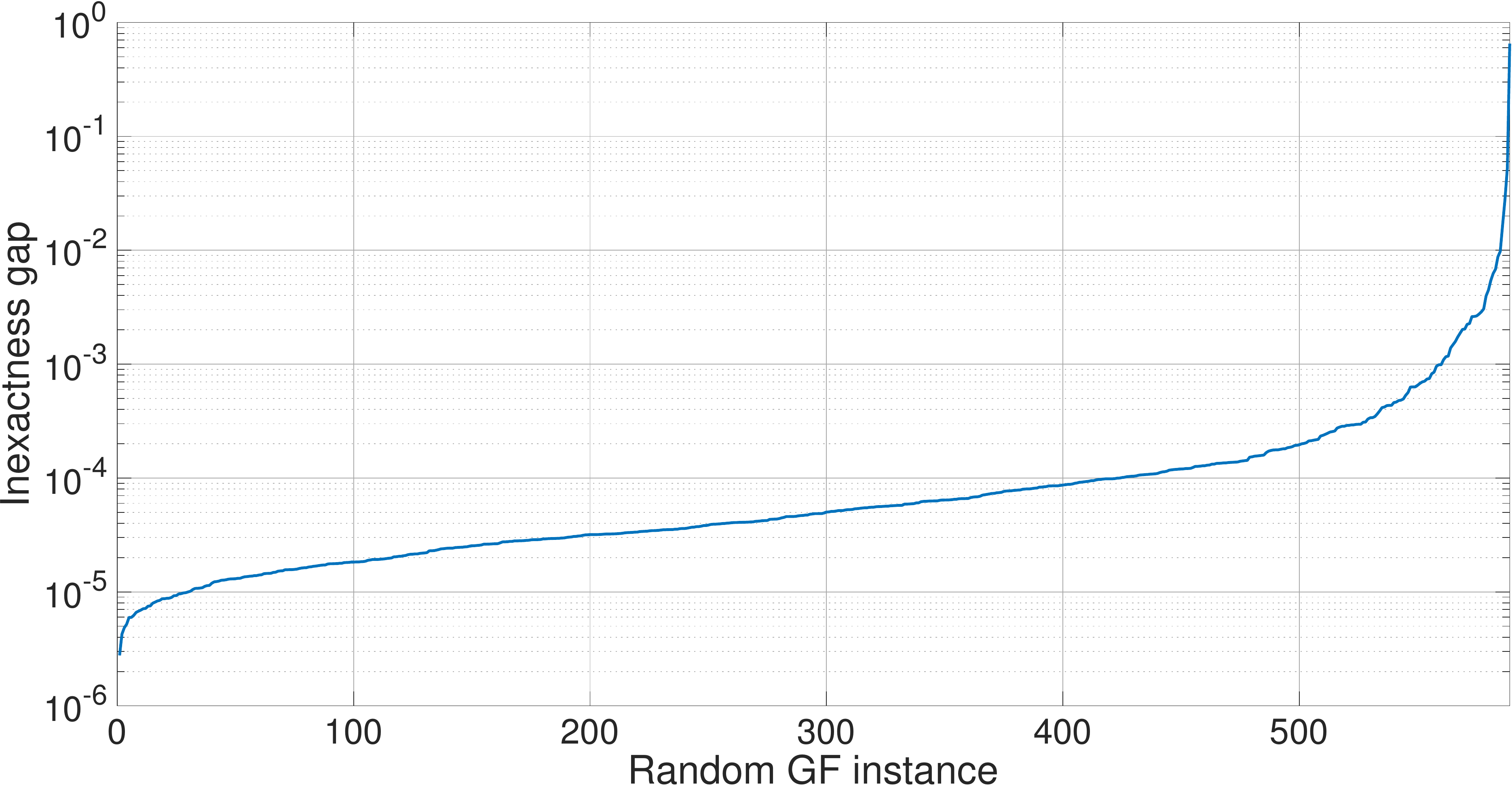}
	\caption{Inexactness gap attained by \eqref{eq:G2} over random feasible GF instances.}
	\label{fig:605feas}
\end{figure}

We first tested the \eqref{eq:G2} solver on the original tree Belgian network. The obtained pressure and flow values agreed with those of \cite{wolf2000energy}. We then inserted additional pipelines on the original benchmark to get a non-radial network as shown in Figure~\ref{fig:belgian}. Even though Theorem~\ref{th:gfexact} requires non-overlapping cycles (Assumption~\ref{as:A2}), the modified network does not comply with this assumption. To use reasonable friction coefficients, for every new line $(m,n)$, the coefficient $a_{mn}$ was set equal to the sum of $a_\ell$'s along the $m-n$ path, yielding $a_{2,5}=0.1936$, $a_{10,14}=0.0439$, $a_{7,12}=0.0419$.

The reference pressure at node $1$ and the compression ratios were kept constant as in \cite{wolf2000energy}. Note that the benchmark gas injections $\bq_o$ lie in the range of $[-15.61, 22.01]$. To test the exactness of the \eqref{eq:G2} relaxation under various conditions, we drew $1,000$ random gas injection vectors $\bq$'s by adding a zero-mean unit-variance deviation on the entries of $\bq_o$. To ensure balanced injections, the injection at node $20$ was set to the negative sum of the remaining injections. 

\begin{figure}[t]
	\centering
	\includegraphics[scale=0.2]{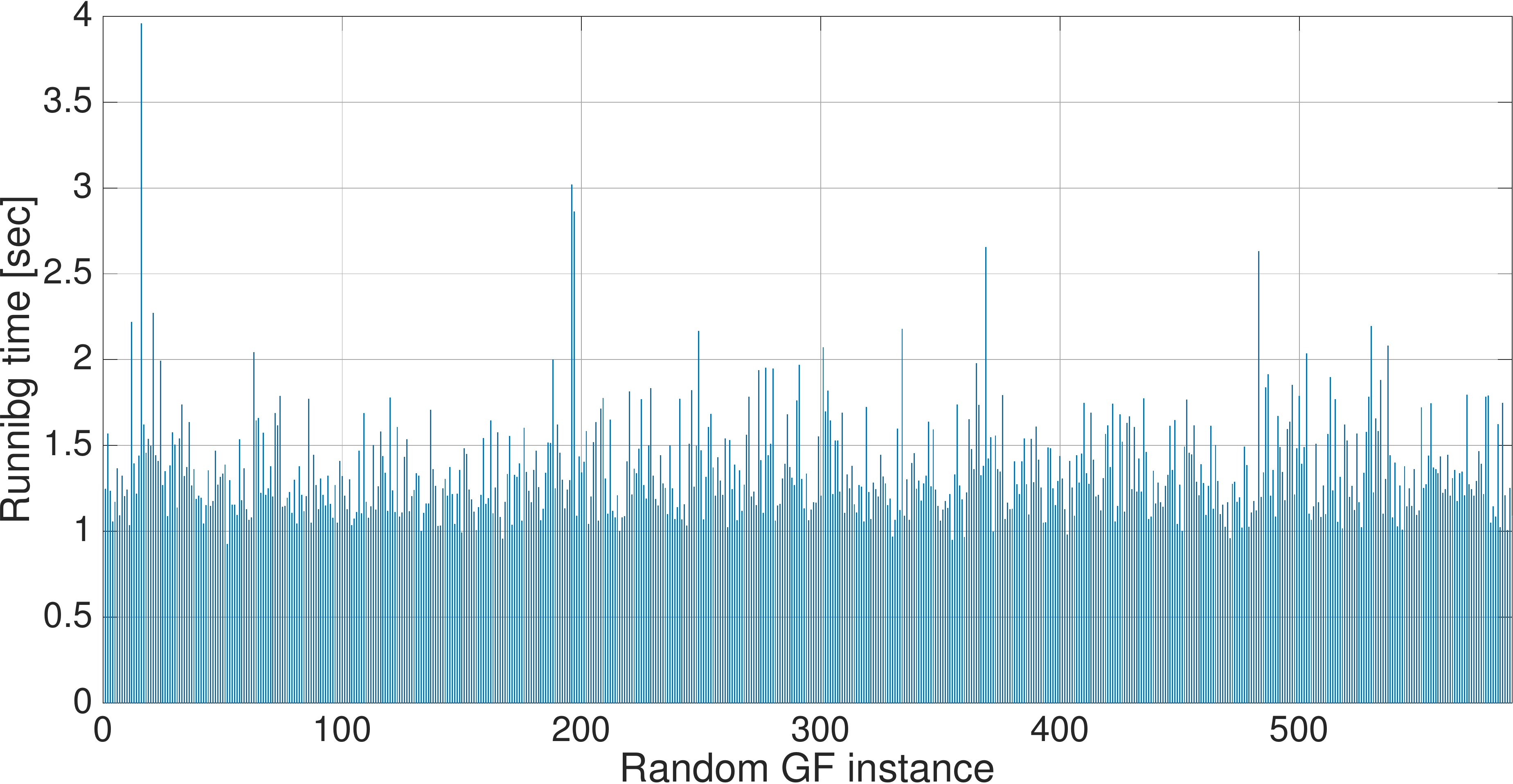}
	\caption{Running time for \eqref{eq:G2} over random feasible GF instances.}
	\label{fig:time}
\end{figure}

Due to the randomness in selecting gas injections, the resulting pipe flows are not guaranteed to satisfy~\eqref{seq:compb}: The problem was found to be infeasible for $411$ out of the $1,000$ random injection vectors. Since \eqref{eq:G2} is a relaxation of \eqref{eq:G1}, these cases are apparently infeasible for \eqref{eq:G1} too. The performance of \eqref{eq:G2} was henceforth tested on the remaining $589$ gas injection instances that were feasible.  

To numerically quantify the success of \eqref{eq:G2} in solving \eqref{eq:G1}, we define the \emph{inexactness gap} for injection vector $\bq$ as
\[g(\bq):=\max_{(m,n)\in\bmcP_a}\frac{|\psi_m-\psi_n|-a_{mn}\phi_{mn}^2}{a_{mn}\phi_{mn}^2}\geq 0\]
for the pressures $\{(\psi_m,\psi_n)\}$ and flows $\{\phi_{mn}\}$ obtained by solving \eqref{eq:G2} for the injection vector $\bq$. Figure~\ref{fig:605feas} depicts the ranked inexactness gap along the feasible GF instances. Based on this curve, the gap was less than $10^{-4}$ for more than $72\%$ of the instances, and less than $10^{-3}$ for more than  $95 \%$. This corroborates that \eqref{eq:G2} performs well even when Assumption~\ref{as:A2} is not met. Figure~\ref{fig:time} shows the running time for solving \eqref{eq:G2} over the $589$ feasible instances. The average (median) running time was $1.39$~sec ($1.28$~sec).

\section{Conclusions}\label{sec:conclusions}
This work has established the uniqueness of the nonlinear steady-state GF equations, put forth a an MI-SOCP-based GF solver, and provided network conditions under which this solver succeeds. Numerical tests have shown that the relaxation is exact. The average time of the solver for a toy $20$-node and $22$-pipe network is less than $2$~seconds; yet its scalability to real-world networks is still to be explored. The success of the relaxation even when the postulated assumptions were not met motivates further research on broadening the conditions. Combining the MI-SOCP-based solver with existing solvers and adopting it to tackle the GF task under the dynamic setup and towards state estimation purposes constitute our current research efforts.

\appendix
\subsection{Proof of Lemma~\ref{le:1cycle}}
Select the cycle direction ${0\rightarrow1\dots\rightarrow k\rightarrow0}$. Without loss of generality (wlog), suppose the $i$-th entry of $\bphi$ corresponds to the flow in the pipe connecting nodes $(i-1)$ and $i$. We will prove the first claim by contradiction; the second claim follows similarly. 
	
	Suppose the flow vectors $\bphi$ and $\tilde{\bphi}$ satisfy \eqref{eq:wey2}--\eqref{eq:comp} and $\sign(\tilde{\bphi}-\bphi)\odot\bn<0$. The assumption $\sign(\tilde{\bphi}-\bphi)\odot\bn<0$ apparently implies that
	\begin{align}\label{eq:signs}
	\phi_i &>\tilde{\phi}_i,~\textrm{if}~n_i=+1, \textrm{and}\\
	\phi_i&<\tilde{\phi}_i,~\textrm{if}~n_i=-1.\nonumber
	\end{align}
	We will show by induction on $i$ that $\tilde{\psi}_i\geq\psi_i$ for all $i\neq 0$. 
	
	Starting with the base step of $i=1$, the edge between nodes $0$ and $1$ can be either a lossy pipe or a compressor. If it is a lossy pipe, denote the RHS of \eqref{seq:wey2a} by $w(\phi_\ell)$. It is not hard to see that $w(\phi_\ell)$ is monotonically increasing in $\phi_\ell$. Depending on the value of $n_1$, two cases can be identified:
	\begin{itemize}
		\item If $n_1=+1$, it follows that
		\[\psi_0-\tilde{\psi}_1=w(\tilde{\phi}_1)<w(\phi_1)=\psi_0-\psi_1\]
		where the two equalities stem from \eqref{seq:wey2a} and the inequality from \eqref{eq:signs}. This implies $\tilde{\psi}_1>\psi_1$. 
		
		\item If $n_1=-1$, we similarly get that
		\[\tilde{\psi}_1-\psi_0=w(\tilde{\phi}_1)>w(\phi_1)=\psi_1-\psi_0\]
		which again implies $\tilde{\psi}_1>\psi_1$.
	\end{itemize}
	If the edge between nodes $0$ and $1$ is a compressor, the linear dependence in \eqref{seq:compa} yields $\tilde{\psi}_1=\psi_1=\alpha_1\psi_0$. Thus, the claim $\tilde{\psi}_1\geq\psi_1$ holds for all cases.
	
	Proceeding with the induction step, we will assume $\tilde{\psi_i}\geq\psi_i$ and prove that $\tilde{\psi}_{i+1}\geq\psi_{i+1}$. If the edge between node $i$ and $i+1$ is a lossy pipe, the following cases arise:
	\begin{itemize}
		\item If $n_{i+1}=+1$, it holds that $\phi_{i+1}>\tilde{\phi}_{i+1}$ from \eqref{eq:signs} and
		\begin{align*}	\tilde{\psi}_i-\tilde{\psi}_{i+1}&=w(\tilde{\phi}_{i+1})<w(\phi_{i+1})=\psi_i-\psi_{i+1}\\
		\implies \tilde{\psi}_{i+1}&>\psi_{i+1}+\tilde{\psi}_i-\psi_i\geq\psi_{i+1}.
		\end{align*}
		\item If $n_{i+1}=-1$, it holds that $\phi_{i+1}<\tilde{\phi}_{i+1}$ from \eqref{eq:signs} and
		\begin{align*}	\tilde{\psi}_{i+1}-\tilde{\psi}_i&=w(\tilde{\phi}_{i+1})>w(\phi_{i+1})=\psi_{i+1}-\psi_i\\
		\implies  \tilde{\psi}_{i+1}&>\psi_{i+1}+\tilde{\psi}_i-\psi_i\geq\psi_{i+1}.
		\end{align*}
	\end{itemize}
	If the edge between node $i$ and $i+1$ is a compressor, we get $\tilde{\psi}_{i+1}=\alpha_{i+1}\tilde{\psi}_i\geq\alpha_{i+1}\psi_i=\psi_{i+1}$. Therefore, the claim $\tilde{\psi}_{i+1}\geq\psi_{i+1}$ holds for all cases. 
	
	The claim $\tilde{\psi}_{i}\geq\psi_{i}$ holds with equality only if all edges from node $0$ to $i$ are compressors. However, this is practically impossible for $i\geq 2$ since every compressor is modeled as an ideal compressor followed by a lossy pipe. Therefore $\tilde{\psi}_{i}>\psi_{i}$ for all $i\geq 2$. Applying the latter around the cycle gives $\tilde{\psi}_0>\psi_0$, which contradicts the hypothesis of fixed $\psi_0$.

\subsection{Proof of Theorem~\ref{th:gfexact}}
Let $(\bphi,\bpsi)$ be the unique solution to \eqref{eq:G1}, and $(\tbphi,\tbpsi)$ a minimizer of \eqref{eq:G2}. Proving by contradiction, suppose there exists an edge $\ell$ for which $\tilde{\phi}_\ell\neq\phi_\ell$. Since both flow vectors satisfy \eqref{eq:mc2}, their difference 
\begin{equation}\label{eq:bn}
\bn:=\tilde{\bphi}-\bphi
\end{equation}
must lie in $\nullspace(\bA^\top)$. The nullspace of $\bA^\top$ is spanned by the indicator vectors of all fundamental cycles in the network graph~\cite[Corollary~14.2.3]{GodsilRoyle}. Therefore, for all edges not belonging to a cycle, the corresponding entries of $\bn$ are zero. Then, the edge $\ell$ must belong to at least one cycle. In fact, by Assumption~\ref{as:A2}, edge $\ell$ belongs to a single cycle, which will be henceforth termed cycle $\mcC$. 

Based on cycle $\mcC$, the remainder of the proof is organized in three steps. The first step constructs a flow vector $\hbphi$ that satisfies constraints \eqref{eq:mc2} and \eqref{seq:compb}. The second step constructs a pressure vector $\hbpsi$ so that the pair $(\hbphi,\hbpsi)$ is feasible for \eqref{eq:G2}. The third step shows that $(\hbphi,\hbpsi)$ attains a smaller objective for \eqref{eq:G2}, thus contradicting the optimality of $(\tbphi,\tbpsi)$. 

Commencing with the first step, define the flow vector $\hat{\bphi}$
\begin{equation*}
\hat{\phi}_\ell:=\left\{\begin{array}{ll}
\phi_\ell &,~\ell\in \mcC\\
\tilde{\phi}_\ell &,~\text{otherwise}.
\end{array}\right.
\end{equation*}
The constructed flow vector $\hbphi$ differs from $\tbphi$ only on cycle $\mcC$. Due to \eqref{eq:bn}, we can write 
\begin{equation}\label{eq:t2}
\hbphi=\tbphi-\lambda\bn^\mcC
\end{equation}
for a $\lambda\neq0$ and where $\bn^\mcC$ is the indicator vector of cycle $\mcC$. Since $\tilde{\bphi}$ satisfies constraint \eqref{eq:mc2} and $\bA^\top\bn^\mcC=\bzero$, then
\[\bA^\top\hat{\bphi}=\bA^\top\tilde{\bphi}-\lambda \bA^\top\bn^\mcC=\bA^\top\tilde{\bphi}=\bq.\]
This proves that $\hbphi$ satisfies constraint \eqref{eq:mc2} as well. Granted there are no compressors in cycles (Assumption~\ref{as:A1}) and since $\tilde{\bphi}$ satisfies constraint \eqref{seq:compb}, then $\hat{\phi}_p=\tilde{\phi}_p\geq 0$ for all compressor edges $p\in\mcP_a$.

We proceed with the second step and construct the pressure vector $\hbpsi$. To define its $i$-th entry, we identify three cases depending on the location of node $i$ relative to cycle $\mcC$.


\emph{a) Node $i\notin \mcC$, and the shortest path between $i$ and $r$ has no edge in $\mcC$.} Define the modified pressure at node $i$ as 
\[\hat{\psi}_i:=\tilde{\psi}_i.\]

\emph{b) Node $i\in \mcC$.} Identify the node $i_\mcC\in \mcC$ with the shortest path to the reference node $r$. Define the modified pressure at node $i$ as
	\[\hat{\psi}_i:=\psi_i+(\tilde{\psi}_{i_\mcC}-\psi_{i_\mcC}).\]
The node $i_\mcC$ may be $i$ itself, implying $\hat{\psi}_{i_\mcC}=\tilde{\psi}_{i_\mcC}$. 

\emph{c) Node $i\notin \mcC$, but the shortest path between $i$ and $r$ has an edge in $\mcC$.} Identify the node  $k\in\mcC$ that is closest to node $i$, that is $d_{i-k}=\min_{j \in \mcC} d_{i-j}$. Note that the nodal pressures from $r$ to $k$ have been defined under cases \emph{a)} and \emph{b)}. Starting from node $k$ and its updated pressure $\hat{\psi}_k$, we next define the constructed pressures along the shortest path from $k$ to $i$ in a sequential fashion. Moving along the $k-i$ path, say the first node is $k+1$ and is connected to $k$ by edge $\ell=(k,k+1)$. Then, we define the pressure at node $k+1$ as
\begin{align*}
\hat{\psi}_{k+1}:=\left\{\begin{array}{ll}
\alpha_\ell \hat{\psi}_k&,\text{ if } \ell\in\mcP_a\\
\hat{\psi}_k+(\tilde{\psi}_{k+1}-\tilde{\psi}_k)&,\text{ if }\ell\in\bmcP_a
\end{array}\right..
\end{align*}
The process is repeated until we reach node $i$. 

As an example, let us construct $\hbpsi$ for Fig.~\ref{fig:example}. Suppose cycle $\mcC_2$ is the cycle over which flows differ. Then, nodes $\{1,2\}$ fall under case \emph{a)}; nodes $\{3,4,5\}$ under case \emph{b)}; and nodes $\{6,7,8,910\}$ under \emph{c)}, with node $5$ acting as node $k$. Then, the constructed pressure vector is
\[\hbpsi=\begin{bmatrix}
\hat{\psi}_1\\
\hat{\psi}_2\\
\hat{\psi}_3\\
\hat{\psi}_4\\
\hat{\psi}_5\\
\hat{\psi}_6\\
\hat{\psi}_7\\
\hat{\psi}_8\\
\hat{\psi}_9\\
\hat{\psi}_{10}
\end{bmatrix}=
\begin{bmatrix}
\tilde{\psi}_1\\
\tilde{\psi}_2\\
\tilde{\psi}_3\\
\psi_4+(\tilde{\psi}_3-\psi_3)\\
\psi_5+(\tilde{\psi}_3-\psi_3)\\
\alpha_{5,6}\hat{\psi}_5\\
\hat{\psi}_6+(\tilde{\psi}_7-\tilde{\psi}_6)\\
\hat{\psi}_7+(\tilde{\psi}_8-\tilde{\psi}_7)\\
\hat{\psi}_8+(\tilde{\psi}_9-\tilde{\psi}_8)\\
\hat{\psi}_9+(\tilde{\psi}_{10}-\tilde{\psi}_9)
\end{bmatrix}.\]

We next show that the constructed $(\hbphi,\hbpsi)$ is feasible for \eqref{eq:G2}. From case \emph{b)}, it follows that for all edges $(m,n)\in \mcC$
\begin{equation}\label{eq:define}
\hat{\psi}_m-\hat{\psi}_n=\psi_m-\psi_n=a_{mn}\sign(\phi_{mn})\phi_{mn}^2.
\end{equation}
The latter implies that constraint \eqref{eq:weyr} is satisfied with equality for all $(m,n)\in \mcC$. Moreover, for all edges $(m,n)\notin\mcC$, cases \emph{b)} and \emph{c)} yield that \[\hat{\psi}_m-\hat{\psi}_n=\tilde{\psi}_m-\tilde{\psi}_n.\] 
Then, since $(\tilde{psi}_m,\tilde{psi}_n)$ satisfy \eqref{eq:weyr}, the same holds for $(\hat{psi}_m,\hat{psi}_n)$ for all $(m,n)\notin\mcC$. The previous two arguments show that $(\hat{\bphi},\hat{\bpsi})$ is feasible for \eqref{eq:G2}.

Continuing with the third step of this proof, note that the objective $r(\bpsi)$ sums up the absolute pressure differences along lossy pipes. Since by construction these differences have changed only along $\mcC$, we get that
\begin{equation}\label{eq:costdif}
r(\tilde{\boldsymbol{\psi}})-r(\hat{\boldsymbol{\psi}})=\sum_{(m,n)\in \mcC}|\tilde{\psi}_m-\tilde{\psi}_n|-|\hat{\psi}_m-\hat{\psi}_n|.
\end{equation}

As the pressure differences depend on flows, we next compare $\hbphi$ and $\tbphi$ using \eqref{eq:t2}. Since the edge directions are assigned arbitrarily, assume wlog that $\hat{\phi}_{mn}\geq 0$ for all $(m,n)\in \mcC$. Given $\bn^\mcC$ and \eqref{eq:t2}, one can find the value of $\lambda$. If $\lambda<0$, reverse the reference direction of cycle $\mcC$ to redefine its indicator vector as $-\bn^\mcC$, and get a positive $\lambda$. Then we then assume $\lambda>0$ wlog.

Recall that $\bn^\mcC\in\{0,\pm 1\}^P$. Define the set of edges in $\mcC$ corresponding to positive entries of $\bn^\mcC$ as $\mcP^+\subseteq \mcC$. Likewise, define the set of edges in $\mcC$ corresponding to negative entries of $\bn^\mcC$ as $\mcP^-\subseteq \mcC$. Take for example cycle $\mcC_1$ in Fig.~\ref{fig:example}: Its edges are grouped as $\mcP^+=\{(3,5),(4,3)\}$ and $\mcP^-=\{(5,4)\}$. From \eqref{eq:t2}, it follows that 
\begin{align*}
0\leq\hat{\phi}_\ell<\tilde{\phi}_\ell&,\quad \forall \ell\in\mcP^+;~\text{ and}\\
\tilde{\phi}_\ell<\hat{\phi}_\ell&,\quad \forall \ell\in\mcP^-.
\end{align*}

Summing up the pressure drops along cycle $\mcC$ for $\hat{\bpsi}$ should be zero. In Figure~\ref{fig:example} for example we have
\[(\hat{\psi}_3-\hat{\psi}_5)+(\hat{\psi}_5-\hat{\psi}_4)+(\hat{\psi}_4-\hat{\psi}_3)=0.\]
Since the pressure drop is positive along the edges in $\mcP^+$, and negative along the edges in $\mcP^-$, it holds that
\begin{align}\label{eq:sumhat}
&\sum_{(m,n)\in\mcP^+}(\hat{\psi}_m-\hat{\psi}_n)=\sum_{(m,n)\in\mcP^-}(\hat{\psi}_m-\hat{\psi}_n)\nonumber\\
\implies~&\sum_{(m,n)\in \mcC}|\hat{\psi}_m-\hat{\psi}_n|=2\sum_{(m,n)\in\mcP^+}(\hat{\psi}_m-\hat{\psi}_n)
\end{align}
where the absolute value is trivial since $\hat{\phi}_{mn}\geq0$ for all $(m,n)\in \mcC$. Referring to Figure~\ref{fig:example}, the equations in \eqref{eq:sumhat} imply $(\hat{\psi}_3-\hat{\psi}_5)+(\hat{\psi}_4-\hat{\psi}_3)=(\hat{\psi}_4-\hat{\psi}_5)$.

To draw similar relations on $\tilde{\boldsymbol{\psi}}$, define the set $\tilde{\mcP}^+\subseteq \mcC$ containing any edge $(m,n)\in \mcC$ such that the flow $\tilde{\phi}_{mn}$ is along the direction of $\bn^\mcC$. Similarly, define $\tilde{\mcP}^-:=\mcC\setminus\tilde{\mcP}^+$. Using the same argument as in \eqref{eq:sumhat} for $\tilde{\boldsymbol{\psi}}$, we obtain
\begin{align}\label{eq:sumtilde}
\sum_{(m,n)\in \mcC}|\tilde{\psi}_m-\tilde{\psi}_n|=2\sum_{(m,n)\in\tilde{\mcP}^+}(\tilde{\psi}_m-\tilde{\psi}_n).
\end{align}
Since the flows in $\hat{\bphi}$ for the edges in $\mcP^+$ are aligned with $\bn^\mcC$  and also $\tilde{\phi}_\ell>\hat{\phi}_\ell$ for these edges, it follows that $\mcP^+\subseteq\tilde{\mcP}^+$. Using this fact in \eqref{eq:sumtilde}, we get that
\begin{align}\label{eq:P+less}
2\sum_{(m,n)\in\mcP^+}(\tilde{\psi}_m-\tilde{\psi}_n)
&\leq 2\sum_{(m,n)\in\tilde{\mcP}^+}(\tilde{\psi}_m-\tilde{\psi}_n)\nonumber\\
&=\sum_{(m,n)\in \mcC}|\tilde{\psi}_m-\tilde{\psi}_n|.
\end{align}

For every edge $\ell=(m,n)\in\mcP^+$, it holds
\begin{align}\label{eq:P+less2}
\hat{\psi}_m-\hat{\psi}_n&= a_\ell\hat{\phi}_\ell^2<  a_\ell\tilde{\phi}_\ell^2\leq  \tilde{\psi}_m-\tilde{\psi}_n
\end{align}
where the equality comes from the definition in \eqref{eq:define}; the first inequality in stems from $\tilde{\phi}_\ell>\hat{\phi}_\ell\geq0$; and the second inequality is from \eqref{eq:weyr}. Summing \eqref{eq:P+less2} over all $\ell\in\mcP^+$, and multiplying by $2$ gives
\begin{align*}
2\sum_{(m,n)\in\mcP^+}(\hat{\psi}_m-\hat{\psi}_n)&<2\sum_{(m,n)\in\mcP^+}(\tilde{\psi}_m-\tilde{\psi}_n)\\
\implies~ \sum_{(m,n)\in \mcC}|\hat{\psi}_m-\hat{\psi}_n|&<\sum_{(m,n)\in \mcC}|\tilde{\psi}_m-\tilde{\psi}_n|
\end{align*}
where the inequality stems from \eqref{eq:sumhat} and \eqref{eq:P+less}. From \eqref{eq:costdif}, this implies $r(\tbpsi)>r(\hbpsi)$ contradicting the optimality of $\tbpsi$.

\balance
\bibliography{myabrv,water,gas}
\bibliographystyle{IEEEtran}
\end{document}